\documentclass[11pt, oneside]{article}   	
\usepackage{amsmath, amsthm, amsfonts, amssymb}
\usepackage{graphicx}
\usepackage[colorlinks,citecolor=blue,urlcolor=blue]{hyperref}
\newcommand\Ex{{\mathbb E}}

\newcommand\chitilde{\tilde{\chi}}
\newcommand\Normal{{\mathcal N}}

\newcommand\cF{{\mathcal F}}

\newcommand\N{{\mathbb N}}

\newcommand\R{{\mathbb R}}
\newcommand\C{{\mathbb C}}

\newcommand\dto{\overset{d}{\to }}

\newcommand\one{{\bf 1}}

\newcommand\cc{{c}}

\newcommand\bra[1]{\langle #1 \rangle}

\DeclareMathOperator{\tr}{tr}

\DeclareMathOperator{\Var}{Var}

\DeclareMathOperator{\Hess}{Hess}

\DeclareMathOperator{\diag}{diag}

\newtheorem{theorem}{Theorem}[section]
\newtheorem{corollary}[theorem]{Corollary}
\newtheorem{lemma}[theorem]{Lemma}

\theoremstyle{definition}

\theoremstyle{remark}
\newtheorem{remark}[theorem]{Remark}

\title{Beta Laguerre ensembles in global regime}
\author{Hoang Dung Trinh\footnote{Faculty of Mathematics Mechanics Informatics, VNU University of Science. Email: thdung.hus@gmail.com} \and Khanh Duy Trinh \footnote{Global Center for Science and Engineering, Waseda University. 
Email: trinh@aoni.waseda.jp}}
\begin{document}
\maketitle

\begin{abstract}

Beta Laguerre ensembles, generalizations of Wishart and Laguerre ensembles, can be realized as eigenvalues of certain random tridiagonal matrices. Analogous to the Wishart ($\beta=1$) case and the Laguerre $(\beta = 2)$ case, for fixed $\beta$, it is known that the empirical distribution of the eigenvalues of the ensembles converges weakly to Marchenko--Pastur distributions, almost surely. The paper restudies the limiting behavior of the empirical distribution but in regimes where the parameter $\beta$ is allowed to vary as a function of the matrix size $N$. We show that the above Marchenko--Pastur law holds as long as $\beta N \to \infty$. When $\beta N \to 2c \in (0, \infty)$, the limiting measure is related to associated Laguerre orthogonal polynomials. Gaussian fluctuations around the limit are also studied.  

\medskip

		
\medskip
	
\end{abstract}

\section{Introduction}
Beta Laguerre ($\beta$-Laguerre) ensembles are ensembles of $N$ positive particles distributed according  the following joint probability density function
\begin{equation}\label{bLE}
	\frac{1}{Z_{N, M}^{(\beta)}}\prod_{i < j}|\lambda_j - \lambda_i|^\beta \prod_{i = 1}^N \left( \lambda_i^{\frac{\beta}{2}(M-N+1) - 1} e^{-\frac{\beta M}{2}  \lambda_i} \right), \quad (\lambda_i > 0),
\end{equation}
where $\beta > 0$ and $M > N -1$ are parameters, and $Z_{N, M}^{(\beta)}$ is the normalizing constant. For two special values $\beta = 1,2$, they are the joint density of the eigenvalues of Wishart matrices and Laguerre matrices, respectively. For general $\beta > 0$, the ensembles can be realized as eigenvalues of a random tridiagonal matrix model $J_N = B_N (B_N)^t$ with a bidiagonal matrix $B_N$ consisting of independent entries distributed as
\[
	B_N = \frac{1}{\sqrt{\beta M}} \begin{pmatrix}
		\chi_{\beta M}	\\
		\chi_{\beta (N-1)}	&\chi_{\beta(M - 1)}	\\
		&\ddots	&\ddots\\
		&&\chi_\beta	&\chi_{\beta(M - N + 1)}
	\end{pmatrix},
\]
where $\chi_k$ denotes the chi distribution with $k$ degrees of freedom, and $(B_N)^t$ denotes the transpose of $B_N$.

Wishart matrices (resp.~Laguerre matrices) are  random matrices of the form $M^{-1}G^{t}G$ (resp.~$M^{-1}G^{*}G$), where $G$ is an $M \times N$ matrix consisting of i.i.d.~(independent identically distributed) entries of standard real (resp.~complex) Gaussian distribution.  Here $G^*$  denotes the  Hermitian conjugate of the matrix $G$. These random matrix models can be defined in a different way as invariant probability measures on the set of symmetric (resp.~Hermitian) matrices. The limiting behavior of their eigenvalues has been studied intensively, and it is known that as $N \to \infty$ with $N/M \to \gamma \in (0,1)$, the empirical distribution 
\[
	L_N = \frac{1}{N}\sum_{i = 1}^N \delta_{\lambda_i}
\]  
converges to the Marchenko--Pastur distribution with parameter $\gamma$, almost surely. Here $\delta_\lambda$ denotes the Dirac measure at $\lambda$. The convergence means that for any bounded continuous function $f$, 
\[
	\int f(x) dL_N(x) = \frac{1}{N} \sum_{i = 1}^N f(\lambda_i) \to \int f(x) mp_\gamma(x) dx\quad \text{as}\quad N \to \infty, \text{almost surely,}
\]
where $mp_\gamma(x)$ is the density of the Marchenko--Pastur distribution with parameter $\gamma$,
\[
	mp_\gamma(x) = \frac{1}{2\pi \gamma x} \sqrt{(\lambda_+ - x)(x - \lambda_-)}, \quad (\lambda_- < x < \lambda_+), \quad \lambda_{\pm} = (1 \pm \sqrt \gamma)^2.
\]
Gaussian fluctuations around the limit were also established with explicit formula for the limiting variance. The convergence to a limit and fluctuations around the limit of the empirical distributions are two typical problems in the global regime when a random matrix model is considered. Refer the book \cite{Pastur-book} for more details on Wishart and Laguerre matrices.

The convergence to Marchenko--Pastur distributions and fluctuations around the limit were extended to $\beta$-Laguerre ensembles for general $\beta> 0$ in \cite{DE06} by using the random tridiagonal matrix model. Based on the model as well, results in the local regime and the edge regime were established \cite{Jacquot-Valko-2011, Sosoe-Wong-2014}. Note that the parameter $\beta$ is fixed in those studies.

The aim of this paper is to refine results in the global regime of $\beta$-Laguerre ensembles. We assume that the parameter $\beta$ varies as a function of $N$, while the parameter $M$ depends on $N$ in a way that $N/M \to \gamma \in (0,1)$ as usual. We show that the Marchenko--Pastur law holds as long as $\beta N \to \infty$. When $\beta N $ stays bounded, the limiting measure is related to associated Laguerre orthogonal polynomials. For the proof, we extend some ideas that were used in \cite{DS15, Nakano-Trinh-2018, Trinh-2017} in case of Gaussian beta ensembles. Our main results are stated in the following two theorems.

\begin{theorem}[Convergence to a limit distribution]\label{thm:LLN}
\begin{itemize}
	\item[\rm(i)] As $\beta N \to \infty$, the empirical distribution $L_N$ converges weakly to the Marchenko--Pastur distribution with parameter $\gamma$, almost surely. Here $L_N = N^{-1} \sum_{i = 1}^N \delta_{\lambda_i}$ is the empirical distribution of the $\beta$-Laguerre ensembles~\eqref{bLE}. 
	
	\item[\rm(ii)] As $\beta N \to 2c \in (0,\infty)$, the sequence $\{L_N\}$ converges weakly to the probability measure $\nu_{\gamma, c}$, almost surely, where $\nu_{\gamma, c}$ is given in Theorem~{\rm\ref{thm:2c-unscale}}.
\end{itemize}
\end{theorem}

To prove the above results, it is sufficient to show that any moment of $L_N$ converges almost surely to the corresponding moment of $\nu_{\gamma, c}$. Here for convenience, we refer to $\nu_{\gamma, \infty}$ as the Marchenko--Pastur distribution with parameter $\gamma$. It then follows that (see \cite{Trinh-ojm-2018}) for any continuous function $f$ of polynomial growth (there is a polynomial $P(x)$ such that $|f(x)| \le P(x)$ for all $x \in \R$), 
\[
	\int f(x) dL_N(x) = \frac{1}{N} \sum_{i = 1}^N f(\lambda_i) \to \int f(x) d\nu_{\gamma, c}( x)\quad \text{as}\quad N \to \infty, \text{almost surely.}
\]   

\begin{theorem}[Gaussian fluctuations around the limit] \label{thm:CLT-full}
Assume that the function $f$ has continuous derivative of polynomial growth. Then the following hold.
	\begin{itemize}
		\item[\rm(i)]	As $\beta N \to \infty$, 
		\[
			\sqrt{\beta} \left( \sum_{i = 1}^N f(\lambda_i) - \Ex\bigg[ \sum_{i = 1}^N f(\lambda_i) \bigg]  \right) \dto \Normal(0, \sigma_f^2),
		\] 
		where 
\[
	\sigma_f^2 = \frac{1}{2\pi^2} \int_{\gamma_-}^{\gamma^+}\int_{\gamma_-}^{\gamma^+} \left(\frac{f(y) - f(x)}{y - x}\right)^2 \frac{4\gamma - (x - \gamma_m)(y - \gamma_m)}{\sqrt{4 \gamma - (x - \gamma_m)^2} \sqrt{4 \gamma - (y - \gamma_m)^2}}dx dy,
\]
with $\gamma_m = (\gamma_{-} + \gamma_{+})/2 = (1 + \gamma)$. Here `$\dto$' denotes the convergence in distribution.
		\item[\rm(ii)] As $\beta N \to 2c \in (0, \infty)$, 
		\[
			\sqrt{\beta} \left( \sum_{i = 1}^N f(\lambda_i) - \Ex\bigg[ \sum_{i = 1}^N f(\lambda_i) \bigg]  \right) \dto \Normal(0, \sigma_{f,c}^2),
		\]
		where $\sigma_{f,c}^2$ is a constant.
	\end{itemize}
\end{theorem}

The paper is organized as follows. In the next section, we introduce the random tridiagonal matrix model and related concepts. Results on convergence to a limit and Gaussian fluctuations around the limit are shown in Section~3 and Section~4, respectively.

\section{Random tridigonal  matrix model and spectral measures}
\subsection{Random tridiagonal matrix model}

Let $G$ be an $M\times N$ random matrix consisting of i.i.d.\ real standard Gaussian random variables. Then $X = M^{-1} G^t G$ is called a Wishart matrix. When  $M \ge N$, the eigenvalues $\lambda_1, \dots, \lambda_N$ of $X$ have the following joint density
\[
	\frac{1}{Z_{M, N}} |\Delta(\lambda)| \prod_{i = 1}^N \left( \lambda_i^{\frac12(M - N + 1) - 1} e^{-\frac M2 \lambda_i} \right),\quad (\lambda_i > 0),
\] 
where $Z_{M,N}$ is the normalizing constant, and $\Delta(\lambda)=\prod_{i < j}(\lambda_j - \lambda_i)$ denotes the Vandermonde determinant. 
A Laguerre matrix $X = M^{-1} G^* G$ is obtained when the entries of $G$ are i.i.d.\ of  standard complex Gaussian distribution, where recall that $G^*$ stands for the Hermitian conjugate of $G$. In this case, the joint density of the eigenvalues has an analogous formula to the Wishart case. Then $\beta$-Laguerre ensembles are defined to be ensembles of $N$ positive particles with the joint density 
\begin{equation}\label{bLE-2}
	\frac{1}{Z_{M, N}^{(\beta)}} |\Delta(\lambda)|^\beta \prod_{i = 1}^N \left( \lambda_i^{\frac\beta2(M - N + 1) - 1} e^{-\frac {\beta M}{2} \lambda_i} \right),\quad (\lambda_i > 0),
\end{equation}
where $\beta > 0$ and $M > N - 1$, which generalizes Wishart $(\beta = 1)$ and Laguerre $(\beta = 2)$ ensembles.

A random tridiagonal matrix model for $\beta$-Laguerre ensembles was introduced in \cite{DE02} based on tridiagonalizing Wishart or Laguerre matrices. Let 
\[
	B_N = \frac{1}{\sqrt{\beta M}} \begin{pmatrix}
		\chi_{\beta M}	\\
		\chi_{\beta (N-1)}	&\chi_{\beta(M - 1)}	\\
		&\ddots	&\ddots\\
		&&\chi_\beta	&\chi_{\beta(M - N + 1)}
	\end{pmatrix}
\]
denote a random bidiagonal matrix with independent entries. Then the eigenvalues of the tridiagonal matrix $J_N = B_N (B_N)^t$ are distributed as the $\beta$-Laguerre ensembles~\eqref{bLE-2}. Using this random matrix model and combinatorial arguments, the convergence to Marchenko--Pastur distributions and Gaussian fluctuations around the limit (only for polynomial test functions) were established in \cite{DE06}.

\subsection{Spectral measures of Jacobi matrices}

A symmetric tridiagonal matrix is called a Jacobi matrix.
Spectral measures of Jacobi matrices $\{J_N\}$ have been studied recently. Here the spectral measure of $J_N$ is defined to be the probability measure $\mu_N$ satisfying 
\[
	\int x^k d\mu_N(x) = (J_N)^k(1,1), \quad k = 0,1,2,\dots.
\]
It follows from the spectral decomposition of $J_N$ that $\mu_N$ has the following form 
\[
	\mu_N = \sum_{i = 1}^N q_i^2 \delta_{\lambda_i},\quad (q_i^2 = v_i(1)^2),
\]
where $v_1, \dots, v_N$ are the normalized eigenvectors of $J_N$ corresponding to the eigenvalues $\lambda_1, \dots, \lambda_N$.

Spectral measures can also be defined for infinite Jacobi matrices. Let $J$ be an infinite Jacobi matrix,
\[
	J = \begin{pmatrix}
		a_1	&b_1	\\
		b_1	&a_2		&b_2\\
		&\ddots	&\ddots	&\ddots
	\end{pmatrix}, \quad a_i \in \R, b_i > 0.
\]
Then there exists a probability measure $\mu$ such that 
\[
	\int x^k d\mu(x) = J^k(1,1), \quad k = 0,1,\dots.
\]
When the above moment problem has a unique solution $\mu$, the measure $\mu$ is called the spectral measure of $J$. A simple, but useful sufficient condition for the uniqueness is given by \cite[Corollary 3.8.9]{Simon-book-2011}
\[
	\sum_{n=1}^\infty \frac{1}{b_n} = \infty.
\]

By definition, moments of the spectral measure $\mu_N$ depend locally on upper left entries of $J_N$, and thus, the limiting behavior of $\mu_N$ follows easily from those of entries. In particular, for fixed $\beta$, as $N \to \infty$ with $N/M \to \gamma \in (0,1)$, entry-wisely, 
\[
	B_N = \frac{1}{\sqrt{\beta M}} \begin{pmatrix}
		\chi_{\beta M}	\\
		\chi_{\beta (N-1)}	&\chi_{\beta(M - 1)}	\\
		&\ddots	&\ddots\\
		&&\chi_\beta	&\chi_{\beta(M - N + 1)}
	\end{pmatrix}
	\to \begin{pmatrix}
		1	\\
		\sqrt{\gamma}	&1	\\
		&\ddots	&\ddots\\
		
	\end{pmatrix}.
\]
Here the convergence holds almost surely and in $L_q$ for any $q \in [1, \infty)$. It follows that almost surely, the spectral measure $\mu_N$ converges weakly to the spectral measure of the following infinite Jacobi matrix 
\begin{equation}\label{MP-gamma}
	MP_\gamma = \begin{pmatrix}
		1\\
		\sqrt{\gamma}		&1\\
		&\ddots 	&\ddots
	\end{pmatrix}
	\begin{pmatrix}
		1	&\sqrt{\gamma}\\
		&1	&\sqrt{\gamma}		\\
		&&\ddots 	&\ddots
	\end{pmatrix}
	=
	\begin{pmatrix}
		1		&\sqrt{\gamma}	\\
		\sqrt{\gamma}		&1+\gamma		&\sqrt{\gamma}\\
		&\sqrt{\gamma}		&1+\gamma		&\sqrt{\gamma}\\
		&&\ddots 	&\ddots 	&\ddots
	\end{pmatrix},
\end{equation}
which is nothing but the Marchenko--Pastur distribution with parameter $\gamma$ \cite{Trinh-ojm-2018}.

For the Jacobi matrix $J_N$, the weights $q_1^2, \dots, q_N^2$ in the spectral measure $\mu_N$ are independent of the eigenvalues and have Dirichlet distribution with parameter $\beta/2$. From which, the empirical distribution $L_N$ and the spectral measure $\mu_N$ have the following relations
\begin{align}
	 \Ex[\bra{L_N, f}] &= \Ex[\bra{\mu_N, f}] , \label{same-mean}\\
	\Var[\bra{L_N, f}] &= \frac{\beta N + 2}{\beta N} \Var[\bra{\mu_N, f}] - \frac{2}{\beta N} \left(\Ex[\bra{\mu_N, f^2}] - \Ex[\bra{\mu_N, f}]^2 \right), \label{relation-of-variance} 
\end{align}
for suitable test functions $f$. Here we use the notation $\bra{\mu, f}$ to denote the integral $\int f d\mu$ for a measure $\mu$ and an integrable function $f$.

We conclude this section by giving several remarks on spectral measures of Jacobi matrices. The spectral measure $\mu$ orthogonalizes the sequence of polynomials $\{P_n\}_{n \ge 0}$ defined by 
\begin{align*}
	&P_0 = 1, \quad P_1 = x - a_1,\\
	&P_{n+1} = x P_n - a_{n + 1}P_n - b_n^2 P_{n - 1}, \quad n \ge 1.
\end{align*}
In a particular case, $a_n =(\alpha + 2n-1)$ and $b_n = \sqrt{n}\sqrt{\alpha + n}$, the sequence of polynomials  $\{L_n\}$ defined by 
\begin{align*}
	&L_0 = 1, \quad L_1 = x - (\alpha + 1),\\
	&L_{n+1} = x L_n - (\alpha + 2n + 1)L_n - n(\alpha + n) L_{n - 1}, \quad n \ge 1,
\end{align*}
coincides with the sequence of scaled generalized Laguerre polynomials. The spectral measure in this case is the Gamma distribution with parameters $(\alpha+1, 1)$, that is, the probability measure with density $\Gamma(\alpha + 1)^{-1} x^\alpha e^{-x}, x > 0$. In other words, the Gamma distribution with parameters $(\alpha + 1,1)$ is the spectral measure of the following infinite Jacobi matrix $J_\alpha$,
\begin{align*}
	J_\alpha &= \begin{pmatrix}
		\alpha + 1	& \sqrt{\alpha + 1} \\
		\sqrt{\alpha + 1} 	&\alpha + 3	& \sqrt{2}\sqrt{ \alpha + 2} 	\\
		&&\ddots	&\ddots	&\ddots 
	\end{pmatrix} \\
	&=	\begin{pmatrix}
		\sqrt{\alpha + 1}	\\
		\sqrt 1	&\sqrt{\alpha + 2}	\\
		%&\sqrt{2}		&\sqrt{\alpha + 3}	\\
		&\ddots	&\ddots	
	\end{pmatrix}
	\begin{pmatrix}
		\sqrt{\alpha + 1}	&\sqrt1	\\
			&\sqrt{\alpha + 2}	&\sqrt{2}	\\
			%&&\sqrt{\alpha + 3}	&\sqrt{3}	\\
			&&\ddots	&\ddots	
	\end{pmatrix}.
\end{align*}

When the entries of $J_\alpha$ are `shifted' by a constant $c \in \R$, the resulting orthogonal polynomials are called associated Laguerre polynomials. The spectral measure in this case was explicitly calculated in \cite{Ismail-et-al-1988} as Model II for associated Laguerre orthogonal polynomials.

\begin{lemma}[\cite{Ismail-et-al-1988}]\label{lem:associated-Laguerre}
For $c > -1$ and $\alpha + c + 1 > 0$, let 
\[
	J_{\alpha, c} = 
	\begin{pmatrix}
		\sqrt{\alpha + c + 1} \\
		\sqrt{c + 1}	& \sqrt{\alpha + c + 2}\\
		&\ddots	&\ddots
	\end{pmatrix}
	\begin{pmatrix}
		\sqrt{\alpha + c + 1} & \sqrt{c + 1}	\\
		&\sqrt{\alpha + c + 2}	&\sqrt{c+2}\\
		&&\ddots	&\ddots
	\end{pmatrix},
\]
and $\mu_{\alpha,c}$ be its spectral measure. Then the density and the Stieltjes transform of $\mu_{\alpha, c}$ are given by
\begin{align*}
	\mu_{\alpha, c}(x) &= \frac{1}{\Gamma(c+1) \Gamma(1+c+\alpha)} \frac{x^{\alpha} e^{-x}}{|\Psi(c, -\alpha; x e^{-i \pi})|^2}, \quad x \ge 0,\\
	S_{\mu_{\alpha, c}}(z) &= \int_0^\infty \frac{\mu_{\alpha, c} (x) dx}{x - z} = \frac{\Psi(c+1, 1 -\alpha; -z)}{\Psi(c, -\alpha; -z)}, \quad z \in \C \setminus \R.
\end{align*}
Here $\Psi(a, b; z)$ is Tricomi's confluent hypergeometric function.
\end{lemma} 

Note that when $\alpha$ is not an integer, an alternative formula for $\Psi(a, b; z)$ could be used
\begin{align*}
	\Psi(c,  - \alpha; xe^{-i \pi}) &=  \frac{\Gamma(\alpha + 1)}{\Gamma(\alpha + c + 1)} \,_1F_1(c; -\alpha; -x) \\
	&\qquad - \frac{\Gamma(-\alpha - 1)}{\Gamma(c)} x^{\alpha + 1} e^{-i\pi \alpha} \,_1F_1 (\alpha + c + 1; 2+\alpha; -x) ,
\end{align*}
where ${}_1F_1(a; b; z)$ is the Kummer function.

\section{Convergence to a limit distribution}
In what follows, the parameter $M$ depends on $N$ in the way that $N/M \to \gamma \in (0,1)$ as $N \to \infty$. We study the limiting behavior of the empirical distribution $L_N$ through its moments $\bra{L_N, x^r}, r=0,1,2,\dots$. Recall from the identity \eqref{same-mean} that for $r = 0,1,2,\dots,$
\[
	\Ex[\bra{L_N, x^r}] = \Ex \bigg[\frac1N \tr [(J_N)^r] \bigg] = \Ex[\bra{\mu_N, x^r}] = \Ex[(J_N)^r(1,1)].
\]
Here $\tr[A]$ denotes the trace of a matrix $A$.

\subsection{The Marchenko--Pastur regime, $\beta N \to \infty$}
A key observation in this regime is the following asymptotic behavior of chi distributions.
\begin{lemma}
As $k \to \infty$,
	\[
		\frac{\chi_k}{\sqrt{k}} \to  1 \quad \text{ in $L^q$ for any $1 \le q < \infty$.}
	\]
	
\end{lemma}

Let $\{c_i\}_{i = 1}^N$ and $\{d_i\}_{i = 1}^{N-1}$ be the diagonal and the sub-diagonal of $B_N$, respectively. Note that although we do not write explicitly, $\{c_i\}$ and $\{d_i\}$ depend on the triple $(N, M, \beta)$.
It is clear that as $\beta N \to \infty$, 
\[
	B_N = \frac{1}{\sqrt{\beta M}} \begin{pmatrix}
		\chi_{\beta M}	\\
		\chi_{\beta (N-1)}	&\chi_{\beta(M - 1)}	\\
		&\ddots	&\ddots\\
		&&\chi_\beta	&\chi_{\beta(M - N + 1)}
	\end{pmatrix}
	\to \begin{pmatrix}
		1	\\
		\sqrt{\gamma}	&1	\\
		&\ddots	&\ddots\\
		
	\end{pmatrix}.
\]
Here the convergence means the pointwise convergence of entries, which holds in $L^q$ for any $q \in [1, \infty)$, that is, for fixed $i$, as $N \to \infty$ with $\beta N \to \infty$,
\begin{equation}\label{Lp-convergence-of-cd}
	c_i \to 1, \quad d_i \to \sqrt{\gamma}	\quad \text{in $L^q$ for $q\in [1, \infty)$.}
\end{equation}

Since $J_N$ is a tridiagonal matrix of the form
\[
	J_N = B_N (B_N)^t = \begin{pmatrix}
		c_1^2	&c_1d_1\\
		c_1d_1	&c_2	^2 + d_1^2 	&c_2d_2\\
		&\ddots	&\ddots	&\ddots\\
		&&c_{N-1}d_{N - 1}	&c_N^2 + d_{N-1}^2
	\end{pmatrix},
\]
it follows that for fixed $r \in \N$, when $N>r$, $(J_N)^r(1,1)$ is a polynomial of $\{c_i, d_i\}_{i \le \frac{r+1}2}$. Let us show some explicit formulae for $(J_N)^r(1,1)$,
\begin{align*}
	J_N(1,1) &= c_1^2,\\
	(J_N)^2(1,1) &= c_1^4 + c_1^2 d_1^2,\\
	(J_N)^3(1,1) &= c_1^6 + 2 c_1^4 d_1^2 + c_1^2 c_2^2 d_1^2 + c_1^2 d_1^4.
\end{align*}
Then the pointwise convergence \eqref{Lp-convergence-of-cd} implies that as $N \to \infty$ with $\beta N \to \infty$,
\[
	(J_N)^r(1,1) \to (MP_\gamma)^r(1,1)		\quad \text{in $L^q$ for any $q \in [1,\infty)$,}
\]
where recall that the Jacobi matrix $MP_\gamma$ is given in \eqref{MP-gamma}. 
Consequently, as $N \to \infty$,
\[
	\Ex[(J_N)^r(1,1) ] \to  (MP_\gamma)^r(1,1).
\]

Recall also that the spectral measure of $MP_\gamma$ is the Marchenko--Pastur distribution with parameter $\gamma$. Therefore,  in this regime, the following convergence of the mean values holds. 
\begin{lemma}\label{lem:mean-infinity}
	For any $r \in \N$, as $N \to \infty$ with $\beta N \to \infty$,
	\[
		\Ex[\bra{L_N, x^r}] = \Ex\left[\frac1N  \tr[ (J_N)^r] \right ]  = \Ex[(J_N)^r(1,1) ] \to \bra{mp_\gamma, x^r}.
	\]
\end{lemma}

\subsection{High temperature regime, $\beta N \to 2c \in (0, \infty)$}

Let 
\[
	B_\infty :=\begin{pmatrix}
		\chitilde_{\frac{2c}\gamma}		\\
		\chitilde_{2c}	&\chitilde_{\frac{2c}\gamma}\\
		&\ddots 	&\ddots
	\end{pmatrix}, \quad 
		\chitilde_k = \frac{1}{\sqrt2} \chi_k,
\] 
be an infinite bidiagonal matrix with independent entries.
Then as $\beta N \to 2c$,
\[
	B_N = \frac{1}{\sqrt{\beta M}} \begin{pmatrix}
		\chi_{\beta M}	\\
		\chi_{\beta (N-1)}	&\chi_{\beta(M - 1)}	\\
		&\ddots	&\ddots\\
		&&\chi_\beta	&\chi_{\beta(M - N + 1)}
	\end{pmatrix}
	\dto \frac{\sqrt{\gamma}}{\sqrt c}B_\infty,
\]
meaning that entries of $B_N$ converge in distribution to the corresponding entries of the infinite matrix $\frac{\sqrt{\gamma}}{\sqrt c}B_\infty$. Recall that $(J_N)^r(1,1)$ is a polynomial of $\{c_i, d_i\}_{i \le \frac{r+1}2}$. Moreover, the entries of $B_N$ and those of $B_\infty$ are independent. Therefore we can deduce that as $\beta N \to 2c$, 
\begin{align*}
	(J_N)^r(1,1) &\dto (J_\infty)^r(1,1),\\
	\Ex[(J_N)^r(1,1) ] &\to \Ex[(J_\infty)^r(1,1) ],
\end{align*}
where
\[
	J_\infty = \frac{\gamma}{c} B_\infty (B_\infty)^t.
\] 
By this approach, we have just shown the existence of the limit of the mean values when $\beta N \to 2c$. However, we are not able to identify the limit directly from $J_\infty$.

To identify the limit, we now extend some ideas that used in \cite{DS15}  in case of Gaussian beta ensembles to show the following.
\begin{theorem}\label{thm:2c-unscale}
	Let $\nu_{\gamma, c}$ be the spectral measure of the following Jacobi matrix 
\[
	\frac{\gamma}{c} 
	 \begin{pmatrix}
		\sqrt{\frac{c}\gamma } \\
		\sqrt{c + 1}	& \sqrt{\frac c{\gamma} + 1}\\
			&\ddots	&\ddots
	\end{pmatrix}
	 \begin{pmatrix}
		\sqrt{\frac{c}\gamma }	&\sqrt{c + 1} \\
		& \sqrt{\frac c{\gamma} + 1}	&\sqrt{c + 2}\\
		&&\ddots	&\ddots
	\end{pmatrix} = \frac \gamma c J_{\frac c\gamma, c}.
\]
Then for any $r \in \N$, as $N \to \infty$ with $\beta N \to 2c$,
\[
	\Ex[\bra{L_N, x^r}] = \Ex \left[ \frac1N \tr [(J_N)^r ] \right] = \Ex[(J_N)^r(1,1) ] \to \bra{\nu_{\gamma, c} , x^r}.
\]
\end{theorem}

Theorem~\ref{thm:2c-unscale} is equivalent to Theorem~\ref{thm:2c-scale} below which states a result for scaled $\beta$-Laguerre ensembles. Let us switch to the scaled version. Let $\tilde J_N := \tilde B_N (\tilde B_N)^t$ be a Jacobi matrix, where 
\[
	\tilde B_N = \begin{pmatrix}
		\chitilde_{2\alpha + 2 + 2\kappa(N-1)}	\\
		\chitilde_{2\kappa (N-1)}	&\chitilde_{2\alpha + 2 + 2\kappa(N-2)}		\\
		&\ddots	&\ddots	\\
		&&\chitilde_{2\kappa}	&\chitilde_{2\alpha + 2}
	\end{pmatrix},
\]
(with $\kappa = \beta / 2 $ and $ \alpha = \frac{\beta}{2}(M - N + 1) - 1 = \kappa(M - N + 1) - 1)$. 
Then the joint density of the eigenvalues of $\tilde J_N$ is proportional to 
\[
	 |\Delta(\tilde\lambda)|^{2\kappa} \prod_{i = 1}^N \left( \tilde\lambda_i^{\alpha} e^{-\tilde\lambda_i} \right),\quad \tilde \lambda_i > 0.
\]

Let $\tilde \mu_N$ be the spectral measure of $\tilde J_N$ and 
\[
m_r(N, \kappa, \alpha) =  \Ex[\bra{\tilde\mu_N, x^r}] = \Ex[(\tilde J_N)^r(1,1)].
\] 
Then $m_r(N, \kappa, \alpha)$ satisfies the following duality relation.
\begin{lemma}[cf.~{\cite[Theorem~2.11]{DE06}}]
The function $m_r(N, \kappa, \alpha)$ is a polynomial with respect to $N, \kappa$ and $\alpha$ and satisfies the following relation
	\[
		m_r(N, \kappa, \alpha) = (-1)^r \kappa^r m_r(- \kappa N, \kappa^{-1}, -\alpha/\kappa).
	\]
\end{lemma}

Based on the above duality relation, we now identify the limit of $m_r(N, \kappa, \alpha) $ in the regime where $\kappa N \to c$. For fixed $N$, it is straightforward to calculate the limit of $\kappa^{-1/2} \tilde B_N$ with parameters $(N, \kappa, a \kappa)$ as $\kappa \to \infty$, where $a$ is fixed,
\[
	\left( \frac{1}{\sqrt{\kappa}} \tilde B_N \right) \to \begin{pmatrix}
		\sqrt{a + N - 1} \\
		\sqrt{N - 1}	& \sqrt{a + N - 2}\\
		&\ddots	&\ddots\\
		&&\sqrt{1}	&\sqrt{a}
	\end{pmatrix} =: D_N(a).
\] 
Here the convergence holds in $L^q$ entry-wisely. Therefore,
\[
	\lim_{\kappa \to \infty} \kappa^{-r} m_r(N, \kappa, a \kappa)  = \lim_{\kappa \to \infty} \Ex[\kappa^{-r}  (\tilde J_N)^r(1,1)] = (D_N(a) D_N(a)^t)^r(1,1).
\]

Next, in viewing of the duality relation, it holds that 
\[
\lim_{N \to \infty, \kappa N \to c} m_r(N, \kappa, \alpha)=	\lim_{\kappa \to \infty} (-1)^r \kappa^{-r} m_r(- c, \kappa, -\alpha \kappa). 
\]
Let us consider the  following infinite matrix by exchanging $N \leftrightarrow -c, a \leftrightarrow -\alpha$ and replacing the sign inside square roots as well,
\[
	W_{\alpha, c}=
	 \begin{pmatrix}
		\sqrt{\alpha + c + 1} \\
		\sqrt{c + 1}	& \sqrt{\alpha + c + 2}\\
		&\sqrt{c + 2}	&\sqrt{\alpha + c + 3} \\
		&&\ddots	&\ddots
	\end{pmatrix}.
\]
Let $J_{\alpha, c} = W_{\alpha, c} W_{\alpha, c}^t$ and $l_r(\alpha, c)= (J_{\alpha, c})^r(1,1)$.
Then it follows that
\[
	\lim_{\kappa N \to c} m_r(N, \kappa, \alpha) = \lim_{\kappa \to \infty} (-1)^r \kappa^{-r} m_r(- c, \kappa, -\alpha \kappa) = l_{r}(\alpha,c).
\]
In conclusion, we have just proved the following result.

\begin{theorem}\label{thm:2c-scale}
Let $\alpha > - 1$ and $\cc \ge 0$ be given. Then in the regime where $ \beta N \to 2 \cc$, 
\[
	\Ex\left[ \frac 1N \tr ((\tilde J_N)^r)\right] = \Ex[(\tilde J_N)^r(1,1)] \to \bra{\mu_{\alpha, c}, x^r},
\]
for any $r \in \N$. Here recall that $\mu_{\alpha, c}$ is the spectral measure of $J_{\alpha, c}$ whose density is given in Lemma~{\rm\ref{lem:associated-Laguerre}}
\end{theorem}

The limiting measure in this regime was calculated in \cite{Allez-Wishart-2013} by a different method.

\subsection{Almost sure convergence}
In this section, we complete the proof of Theorem~\ref{thm:LLN} by showing the following almost sure convergence.
\begin{lemma}
For any $r \in \N$, as $N\beta \to 2c \in (0, \infty]$,
\[
	S_N := \frac{1}{N} \tr[(J_N)^r] - \Ex\left [\frac{1}{N} \tr[(J_N)^r] \right] \to 0,\quad \text{almost surely.}
\]
\end{lemma}
The idea is that for fixed $r$, $p_i := (J_N)^r(i, i)$ is independent of $p_j=(J_N)^r(j, j)$, if $|i - j| \ge D_r$, where $D_r$ is a constant. Then write $S_N$ as a sum of $D_r$ sums of independent random variables
\[
	S_N = \frac{1}{N}\sum_{i} (p_{1+i D_r} - \Ex[p_{1+i D_r}])   + \cdots + \frac{1}{N}\sum_{i} (p_{D_r+i D_r} - \Ex[p_{D_r+i D_r} ]).  
\]
For each sum of independent random variables, we use the following result whose proof can be found in the proof of Theorem 2.3 in \cite{Trinh-2017} to show the almost sure convergence.
\begin{lemma}
Assume that for each $N$, the random variables $\{\xi_{N, i}\}_{i = 1}^{\ell_N}$ are independent and that
\begin{equation}\label{4-moment}
	\sup_{N} \sup_{1\le i \le \ell_N} \Ex[(\xi_{N, i})^4] < \infty.
\end{equation}
Assume further that $\ell_N / N \to const \in (0, \infty)$ as $N \to \infty$. Then as $N \to \infty$,
\[
	\frac{1}N \sum_{i = 1}^{\ell_N} \left(\xi_{N, i} - \Ex[\xi_{N, i}] \right) \to 0, \quad \text{almost surely}.
\] 
\end{lemma}

The moment condition \eqref{4-moment} can be easily checked for $p_i$ in the regime $\beta N \to 2c \in (0, \infty]$. The desired almost sure convergence then follows immediately.

\section{Gaussian fluctuations around the limit}

\subsection{Polynomial test functions}

In this section,
we establish central limit theorems (CLTs) for polynomial test functions. Since arguments are similar to those used in \cite{Trinh-2017} in case of Gaussian beta ensembles, we only sketch main steps. Without loss of generality, assume that $M = N/\gamma$, where $\gamma \in (0,1)$ is fixed.

Let us write $B_N$ as
\[
	B_N = \frac{\sqrt{\gamma}}{\sqrt{\beta N}}\begin{pmatrix}
		c_1	\\
		d_1	&c_2	\\
		&\ddots	&\ddots\\
		&&d_{N - 1}	&c_N
	\end{pmatrix} ,\quad \text{where }
	\begin{cases}
		c_i \sim \chi_{\frac{\beta N}{\gamma} - \beta(i - 1)} ,\\
		d_i \sim \chi_{\beta N -  \beta i}.
	\end{cases}
\]
Then 
\[
	J_N = B_N (B_N)^t = \frac{\gamma}{\beta N}\begin{pmatrix}
		c_1^2	&c_1d_1\\
		c_1d_1	&c_2	^2 + d_1^2 	&c_2d_2\\
		&\ddots	&\ddots	&\ddots\\
		&&c_{N-1}d_{N - 1}	&c_N^2 + d_{N-1}^2
	\end{pmatrix} .
\]
Recall that for fixed $r \ge 1$,  the $r$th moment $\bra{\mu_N, x^r}$ is a polynomial of $\{c_i, d_i\}_{1 \le i \le \frac{r+1}{2}}$. It is actually a polynomial of $\{c_i^2, d_i^2\}$, that is, 
\[
	\bra{\mu_N, x^r} = (J_N)^r(1,1) =  \frac{\gamma^r}{(\beta N)^r} \sum_{\vec\eta, \vec\zeta} a(\vec\eta, \vec\zeta)  \prod_{i=1}^r c_i^{2\eta_i} d_i^{2\zeta_i},
\]
where non-negative integer vectors $\vec\eta = (\eta_1, \dots, \eta_r)$ and $\vec\zeta = (\zeta_1, \dots, \zeta_r)$ satisfy $\sum_{i = 1}^r (\eta_i + \zeta_i) = r$. Then, from formulae for moments of chi-squared distributions, we conclude that 
\begin{lemma}
\begin{itemize}
\item[\rm(i)] For $r \in \N$,
\[
	\Ex[\bra{\mu_N, x^r}] = \sum_{k = 0}^r \frac{R_{r; k} (\beta)}{(\beta N)^k},
\]
where $R_{r; k} (\beta)$ is a polynomial in $\beta$ of degree at most $k$.

\item[\rm(ii)] For $r, s \in \N$,
\[
	\Ex[\bra{\mu_N, x^r}\bra{\mu_N, x^s}] = \sum_{k = 0}^{r+s} \frac{Q_{r,s; k} (\beta)}{(\beta N)^k},
\]
where $Q_{r,s; k} (\beta)$ is a polynomial in $\beta$ of degree at most $k$.
\end{itemize}
\end{lemma}

Let $p$ be a polynomial of degree $m$. From the above expressions, we can derive a general form of $\Var[\bra{\mu_N, p}]$, and then that of $\Var[\bra{L_N, p}]$ by taking into account of the relation \eqref{relation-of-variance}. Similar to the case of Gaussian beta ensembles, the formula for $\Var[\bra{L_N, p}]$ can be simplified as follows. The proof is omitted.
\begin{lemma}
Let $p$ be a polynomial of order $m$. Then the variance $\Var[\bra{L_N, p}] $ can be expressed as
\[
	\Var[\bra{L_N, p}]  = \sum_{k = 2}^{2m + 1} \frac{\beta \ell_{p; k} (\beta)}{(\beta N)^k},
\]
where $ \ell_{p; k} (\beta)$ is a polynomial in $\beta$ of degree at most $(k -2)$.
\end{lemma}

\begin{corollary}
\begin{itemize}
\item[\rm(i)]
	As $\beta N \to \infty$,
	\[
		\beta N^2 \Var[\bra{L_N, p}]  \to \sigma_p^2.
	\]
\item[\rm(ii)] As $\beta N \to 2c$,
	\[
		\beta N^2 \Var[\bra{L_N, p}]  \to \sigma_{p,c}^2.
	\]
\end{itemize}
\end{corollary}

Theorem 3.4 in \cite{Trinh-2017} provides sufficient conditions under which CLTs for $\{\bra{L_N, p}\}$ hold in case of Jacobi matrices with independent entries. The result can be easily extended to Wishart-type Jacobi matrices $\{J_N\}$ here by considering the filtration $\{\cF_k = \sigma\{c_i, d_i : i = 1,\dots, k\}\}_k$. The convergence of variances as in the previous corollary is one of the two sufficient conditions. The remaining one is quite similar to that in the Gaussian beta ensembles case, and hence, we will not mention it in details here. Consequently, the following CLTs for polynomial test functions follow. 
\begin{theorem}
Let $p$ be a polynomial. Then the following hold.
\begin{itemize}
	\item[\rm(i)]	As $\beta N \to \infty$, 
	\[
		\sqrt{\beta} N (\bra{L_N, p} - \Ex[\bra{L_N, p}] ) \dto \Normal(0, \sigma_p^2).
	\]
	\item[\rm(ii)]	As $\beta N \to 2c$, 
	\[
		\sqrt{\beta} N (\bra{L_N, p} - \Ex[\bra{L_N, p}] ) \dto \Normal(0, \sigma_{p,c}^2).
	\]	
\end{itemize}
\end{theorem}

\begin{remark}
\begin{itemize}
\item[(i)]
For fixed $\beta$, it was shown in \cite{DE06} that as $N \to \infty$, 
\[
	N \bra{L_N, p} - N\bra{mp_\gamma, p} -  \left( \frac2\beta - 1\right)\bra{\mu_1, p}  \dto \Normal(0, \sigma_p^2/\beta),
\]
where $\mu_1$ is given by
\[
\mu_1=
		\frac14 \delta_{\lambda_{-}} +	\frac14 \delta_{\lambda_{+}} + \frac{1}{2\pi} \frac{1}{\sqrt{(\lambda_+ - x) (x - \lambda_-)}} \one_{(\lambda_-, \lambda_+) } (dx).
\]

\item[(ii)] Using results in case $\beta = 1,2$, we deduce that the limiting variance in the regime $\beta N \to \infty$ is given by (cf.~\cite[Theorem 7.3.1]{Pastur-book})
\[
	\sigma_f^2 = \frac{1}{2\pi^2} \int_{\gamma_-}^{\gamma^+}\int_{\gamma_-}^{\gamma^+} \left(\frac{f(y) - f(x)}{y - x}\right)^2 \frac{4\gamma - (x - \gamma_m)(y - \gamma_m)}{\sqrt{4 \gamma - (x - \gamma_m)^2} \sqrt{4 \gamma - (y - \gamma_m)^2}}dx dy,
\]
where $\gamma_m = (\gamma_{-} + \gamma_{+})/2 = (1 + \gamma)$.
\end{itemize}
\end{remark}

\subsection{$C^1$ test functions}

To extend CLTs from polynomial test functions to  $C^1$ test functions, one idea is to use a type of Poincar\'e inequality.  Consider (scaled) $\beta$-Laguerre ensembles with the joint density proportional to
\[
	 |\Delta(\lambda)|^\beta \prod_{i = 1}^N \left( \lambda_i^{\alpha} e^{-\eta\lambda_i} \right),
\]
where $\alpha = \frac{\beta}{2}(M - N + 1) - 1$, and $\eta > 0$. Directly from the joint density, we can derive a Poincar\'e inequality by using the following result. However, this approach requires $\alpha > 0$. 
\begin{lemma}
[{\cite[Proposition~2.1]{Bobkov-Ledoux-2000}}]
\label{lem:Poincare}
Let $d\nu = e^{-V} dx$ be a probability measure supported on an open convex set $\Omega \subset \R^N$. Assume that $V$ is twice continuously differentiable and strictly convex on $\Omega$. Then for any locally Lipschitz function $F$ on $\Omega$,
\[
	\Var_\nu[F] = \int F^2 d\nu - \left(\int F d\nu \right)^2 \le \int (\nabla F)^t\Hess(V)^{-1} \nabla F d\nu.
\]
Here $\Hess(V)$ denotes the Hessian of $V$.
\end{lemma}

Let $\Omega = \{(\lambda_1, \dots, \lambda_N): 0 < \lambda_1 <\cdots < \lambda_N\} \subset \R^N$. Let $\nu$ be the distribution of the ordered eigenvalues of (scaled) $\beta$-Laguerre ensembles, that is, the probability measure on $\Omega$ of the form
\[
	d\nu = const \times |\Delta(\lambda)|^\beta \prod_{i = 1}^N \left( \lambda_i^{\alpha} e^{-\eta\lambda_i} \right) d\lambda_1 \cdots d\lambda_N= e^{-V} d\lambda_1 \cdots d\lambda_N,
\]
where 
\[
	V = const + \eta\sum_{i = 1}^N \lambda_i - \alpha \sum_{i = 1}^N \log \lambda_i - \frac{\beta}{2} \sum_{i \neq j} \log|\lambda_j - \lambda_i|.
\]
Then the Hessian matrix of $V$ can be  easily calculated 
\begin{align*}
	\frac{\partial^2 V}{\partial \lambda_i^2} &= \frac{\alpha}{\lambda_i^2} + \beta \sum_{j \neq i} \frac{1}{(\lambda_j - \lambda_i)^2},\\
	\frac{\partial^2 V}{\partial \lambda_i \partial \lambda_j} &= -\beta \frac{1}{(\lambda_j - \lambda_i)^2}.
\end{align*}
Observe that $\Hess (V) \ge D= \diag(\frac{\alpha}{\lambda_i^2})$. Here for two real symmetric matrices $A$ and $B$, the notation $A \ge B$ indicates that $A-B$ is positive semi-definite. It follows that $\Hess (V)^{-1} \le D^{-1}$. And hence, using Lemma~\ref{lem:Poincare} with 
\[
	F = \frac{1}{N} \sum_{i = 1}^N f(\lambda_i),
\]
for a continuously differentiable function $f$,
we get the following inequality
\[
	\Var_\nu[F] \le \int ( \nabla F)^t\Hess(V)^{-1}\nabla F d\nu \le \int  (\nabla F)^t D^{-1}\nabla F d\nu.
\]
The inequality can be rewritten as
\begin{equation}
	\Var[\bra{L_N, f}] \le \frac{1}{\alpha N} \Ex[\bra{L_N, (\lambda f'(\lambda))^2}].
\end{equation}
This is one type of Poincar\'e inequality for $\beta$-Laguerre ensembles in this paper.

The restriction of the above inequality is that $\alpha$ must be positive, which does not hold in the regime $\beta N \to 2c$ when $c$ is small. The second approach based on the random Jacobi matrix model removes such restriction. We will end up with a slightly different inequality. Let us begin with several concepts.

A real random variable $X$ is said to satisfy a Poincar\'e inequality if there is a constant $c>0$ such that for any smooth function $f \colon \R \to \R$,
\[
	\Var[f(X)]  \le c \Ex[f'(X)^2].
\]
Here, by smooth we mean enough regularity so that the considering terms make sense. By definition, it is clear that $X$ satisfies a Poincar\'e inequality with a constant $c$, if and only if $\alpha X$ satisfies a Poincar\'e inequality with a constant $c\alpha^2$, for non-zero constant $\alpha$.

\begin{lemma}
	The chi distribution $\chi_k$ satisfies a Poincar\'e inequality with $c = 1$, that is,
	\[
		\Var[f(X)] \le \Ex[f'(X)^2], \quad X \sim \chi_k.
	\]
\end{lemma}
\begin{proof}
The probability density function of the chi distribution with $k$ degrees of freedom is given by
	\[
		\frac{1}{2^{\frac k2 - 1} \Gamma(\frac k2)} x^{k - 1} e^{-\frac{x^2}{2}}, \quad (x > 0).
	\]
Thus, for $k \ge 1$, the conclusion follows immediately from Lemma~\ref{lem:Poincare}.  
	
	Next, we consider the case $0 < k <1$. Let $Y = X^{k}$. Then the probability density function of $Y$ is given by 
	\[
		const\times \exp({-\frac{y^{\frac 2k}}{2}}),\quad ( y > 0).
	\] 
By using Lemma~\ref{lem:Poincare} with $V = const + \frac{y^{\frac 2k}}2$, we obtain that 
	\[
		\Var[g(Y)] \le \frac{k^2}{2-k}\Ex[Y^{2 - \frac 2k} g'(Y)^2].
	\]

For given $f(x)$, let $g(y) = f(y^\frac1k)$. Then 
we see that 
\[
	\Var[f(X)] = \Var[g(Y)] \le  \frac{k^2}{2-k}\Ex[Y^{2 - \frac 2k} g'(Y)^2] = \frac{1}{2-k}\Ex[f'(X)^2] \le \Ex[f'(X)^2].
\]
This means that $X$ satisfies a Poicar\'e inequality with a constant $c= 1$. The proof is complete.
\end{proof}

We need the following property. 
\begin{lemma}[{\cite[Corollary 5.7]{Ledoux-book}}]
Assume that $X_i, (i = 1, \dots, k),$ satisfy Poicar\'e inequalities with constants $c_i$. Assume further that $\{X_i\}$ are independent. Then for any smooth function $g \colon \R^k \to \R$, 
\[
	\Var[g(X_1, \dots, X_k)] \le (\max_i c_i)  \Ex[|\nabla g(X_1, \dots, X_k)|^2].
\]
\end{lemma}

Let $Y = (y_{mn})$ be an $M \times N$ real matrix, and $X = Y^t Y = (x_{ij})$. For a continuously differentiable function $f \colon \R \to \R$, let
\[
	g = g((y_{mn})) = \tr (f(X)).
\]
Then the partial derivatives of $g$ can be expressed as follows.

\begin{lemma}[cf.~{Eq.~(7.2.5) in \cite{Pastur-book}}]
It holds that
	\[
		\left( \frac{\partial g}{\partial y_{mn}}\right)_{M\times N} = 2Y f'(X).
	\]
Consequently 
\[
	\sum_{m,n} \left( \frac{\partial g}{\partial y_{mn}}\right)^2 = 4\tr(Y f'(X) f'(X) Y^t ) = 4\tr(X f'^2(X)) = 4 \sum_{i = 1}^N \lambda_i f'(\lambda_i)^2.
\]
Here $\lambda_1, \dots, \lambda_N$ are the eigenvalues of $X$.
\end{lemma}

%\begin{theorem}
%Assume that $\{\xi_{mn}\}_{1 \le m \le M, 1 \le n \le  N}$ are independent real random variables and satisfy Poincar\'e inequalities with constant $c$. Let $X = Y^t Y$, where $Y = (\xi_{mn})$, and let $\{\lambda_i\}_{i = 1}^N$ be the eigenvalues of $X$.
%Then for any continuously differentiable function $f$, 
%	\[
%		\Var \bigg[\sum_{i = 1}^N f(\lambda_i) \bigg] \le 4c \Ex\bigg[\sum_{i = 1}^N \lambda_i f'(\lambda_i)^2\bigg].
%	\]
%\end{theorem}

Combining three lemmas above, we arrive at another type of Poincar\'e inequality for $\beta$-Laguerre ensembles.
\begin{theorem}  
Assume that $f$ is a continuously differentiable function. Then for $\beta$-Laguerre ensembles~\eqref{bLE}, the following inequality holds
\begin{equation}\label{Poincare-for-Laguerre}
	\Var[\bra{L_N, f}] \le \frac{4}{\beta M N} \Ex[\bra{L_N, x f'(x)^2}].
\end{equation}
\end{theorem}

The above Poincar\'e inequality is a key ingredient to  extend CLTs to $C^1$ test functions. Let us now prove Theorem~\ref{thm:CLT-full}.

%\begin{lemma}
%	Assume that $f$ has continuous derivative of polynomial growth. Then in the regime where $\beta N \to 2c \in (0, \infty]$, 		\[
%			\sqrt{\beta} \left( \sum_{i = 1}^N f(\lambda_i) - \Ex\bigg[ \sum_{i = 1}^N f(\lambda_i) \bigg]  \right) \dto \Normal(0, \sigma_{f,c}^2).
%		\]
%\end{lemma}
\begin{proof}[Proof of Theorem~{\rm\ref{thm:CLT-full}}]
Assume that $f$ has continuous derivative $f'$ of polynomial growth. This implies that $f' \in L^2 ((1+x^2)d\nu_{\gamma, c}(x))$. Recall that for convenience,  $\nu_{\gamma, \infty}$ denotes the Marchenko--Pastur distribution with parameter $\gamma$.
We need the following property of measures determined by moments (called M.~Riesz's Theorem (1923) in \cite{Bakan-2001}): the measure $\mu$ is determined by moments, if and only if polynomials are dense in $L^2((1+x^2)d\mu(x))$. Consequently, there is a sequence of polynomials $\{p_k\}$ converging to $f'$ in $ L^2((1+x^2)d\nu_{\gamma,c}(x))$. It then follows that   
\[
	\int x(p_k - f')^2 d\nu_{\gamma, c}(x) \le \frac12\int (p_k - f')^2 (1+x^2) d\nu_{\gamma, c}(x) \to 0 \quad \text{as}\quad k \to \infty. 
\]

Let $P_k$ be a primitive of $p_k$. Since $P_k$ is a polynomial, for fixed $k$, as $N \to \infty$, 
\begin{equation}\label{CLT-XNk}
	X_{N, k}:= \sqrt{\beta} N (\bra{L_N, P_k} - \Ex[\bra{L_N, P_k}]) \dto \Normal(0, \sigma_{P_k, c}^2), \quad \Var[X_{N, k}] \to  \sigma_{P_k, c}^2.
\end{equation}
Let $Y_N = \sqrt{\beta} N (\bra{L_N, f} - \Ex[\bra{L_N, f}])$. Then by the Poincar\'e inequality~\eqref{Poincare-for-Laguerre},
\begin{equation}\label{Variance-approx}
	\Var[ Y_N - X_{N, k} ] \le \frac{4N}{ M} \Ex[\bra{L_N, x(f' - p_k)^2}],
\end{equation}
which implies that 
\[
	\lim_{k \to \infty} \limsup_{N \to \infty} \Var[ Y_N - X_{N, k} ] \le 4 \gamma  \lim_{k \to \infty} \bra{\nu_{\gamma, c}, x(f' - p_k)^2} = 0.
\]
Here we have used the property that for continuous function $g$ of polynomial growth, 
\[
	\Ex[\bra{L_N, g}] \to \bra{\nu_{\gamma, c}, g}.
\]
Then, the CLT for $Y_N$ follows from the equations \eqref{CLT-XNk} and \eqref{Variance-approx} by using the following general result whose proof can be found in \cite{Nakano-Trinh-2018} or \cite{Trinh-2018-CLT}. The proof is complete.
\end{proof}
\begin{lemma}\label{lem:CLT-triangle}
Let $\{Y_N\}_{N=1}^\infty$ and $\{X_{N,k}\}_{N,k=1}^\infty$ be mean zero real random variables. Assume that
\begin{itemize}
	\item[\rm(i)] for each $k$, as $N \to \infty$,
		$	X_{N,k} \dto \Normal(0, \sigma_k^2),$ and $\Var[X_{N, k}] \to \sigma_k^2$;
	\item[\rm(ii)] 
		$
			\lim_{k \to \infty} \limsup_{N \to \infty} \Var[X_{N,k} - Y_N] =0.
		$
\end{itemize}
Then the limit $\sigma^2 = \lim_{k \to \infty} \sigma_k^2$ exists, and as $N \to \infty$,
\[
 	Y_N \dto \Normal(0, \sigma^2), \quad \Var[Y_N] \to \sigma^2.
\]
\end{lemma}

\bigskip
\noindent{\bf Acknowledgment.}
This work is supported by the VNU University of Science under project number TN.18.03 (H.D.T) and by JSPS KAKENHI Grant Number JP19K14547 (K.D.T.).

%
%\bibliographystyle{spmpsci}
%\bibliography{bib}

\end{document}